\documentclass[11pt]{amsart}
\usepackage{amsmath}
\theoremstyle{plain}
\newtheorem{theorem}{Theorem}[section]
\newtheorem{lemma}[theorem]{Lemma}

\theoremstyle{definition}
\newtheorem{remark}[theorem]{Remark}

\numberwithin{equation}{section}

\def\be{\begin{equation}}
\def\ee{\end{equation}}

\begin{document}

\title[Optimal regularity of Minimal Graphs]
{Optimal regularity of Minimal Graphs\\ in the Hyperbolic Space}
\author{Qing Han}
\address{Beijing International Center for Mathematical Research\\
Peking University\\
Beijing, 100871, China} \email{qhan@math.pku.edu.cn}
\address{Department of Mathematics\\
University of Notre Dame\\
Notre Dame, IN 46556} \email{qhan@nd.edu}
\author{Weiming Shen}
\address{School of Mathematical Sciences\\
Peking University\\
Beijing, 100871, China}
\address{Beijing International Center for Mathematical Research\\
Peking University\\
Beijing, 100871, China} \email{wmshen@pku.edu.cn}
\author{Yue Wang}
\address{School of Mathematical Sciences\\
Peking University\\
Beijing, 100871, China}
\address{Beijing International Center for Mathematical Research\\
Peking University\\
Beijing, 100871, China} \email{1201110027@pku.edu.cn}
\begin{abstract}
We discuss the global regularity of solutions $f$ to the Dirichlet problem 
for minimal graphs in the hyperbolic space 
when the boundary of the domain $\Omega\subset\mathbb R^n$ has a nonnegative mean curvature and 
prove an optimal regularity $f\in  C^{\frac{1}{n+1}}(\bar{\Omega})$.
We can improve the H\"older exponent for $f$ if certain combinations of principal curvatures of the boundary 
do not vanish, a phenomenon observed by F.-H. Lin. 
\end{abstract}

\thanks{The first author acknowledges the support of NSF
Grant DMS-1404596. }
\maketitle

\section{Introduction}\label{sec-Intro}

Anderson \cite{Anderson1982Invent}, \cite{Anderson1983}
studied complete area-minimizing submanifolds and proved that, for any given closed
embedded $(n-1)$-dimensional submanifold $N$ at the infinity of $\mathbb{H}^{n+1}$, there exists a
complete area minimizing integral $n$-current which is asymptotic to $N$ at infinity.
Hardt and Lin \cite{Hardt&Lin1987} discussed the $C^1$-boundary regularity of such
hypersurfaces. Subsequently, Lin \cite{Lin1989Invent} studied the higher order boundary regularity
for solutions to the Dirichlet problem for minimal graphs in the hyperbolic space. 

Assume that $\Omega\subset \mathbb{R}^{n}$ is a bounded domain. 
Lin \cite{Lin1989Invent} studied the Dirichlet problem of the form
\begin{align}\label{eq-main}
\begin{split}
\Delta f-\frac{f_{i}f_{j}}{1+|\nabla f|^{2}}f_{ij}+\frac{n}{f}&=0 \quad\text{in }\Omega,\\
f&=0 \quad \text{on }\partial\Omega,\\
f&>0 \quad\text{in }  \Omega.
\end{split}
\end{align} 
We note that the
equation in \eqref{eq-main} is a quasilinear non-uniformly elliptic
equation. It becomes singular on $\partial\Omega$ since $f = 0$ there. 
Lin \cite{Lin1989Invent} proved that  
\eqref{eq-main} admits a unique solution 
$f\in C(\bar{\Omega})\bigcap C^{\infty}(\Omega)$
if
$\Omega\subset \mathbb{R}^{n}$
is a $C^2$-domain with a nonnegative boundary mean curvature $H_{\partial\Omega} \geq 0$ 
with respect to the inward
normal direction of $\partial\Omega$. Moreover,
the graph of $f$ is a complete minimal hypersurface in the hyperbolic space $\mathbb{H}^{n+1}$ with
the asymptotic boundary $\partial\Omega$. 
Concerning the higher global regularity, Lin proved 
$f\in C^{1/2}(\bar{\Omega})$
if $H_{\partial\Omega} > 0$.  He also expected certain relations between the H\"older exponents for $f$ 
and the vanishing order
of $H_{\partial\Omega}$ at boundary points. (See Remark 3.7 \cite{Lin1989Invent}.)

The primary goal of this paper is to discuss the global regularity of the solution $f$ of \eqref{eq-main}. 
We first discuss the optimal regularity of $f$ in the general case $H_{\partial\Omega} \geq 0$ and prove 
$f\in  C^{\frac{1}{n+1}}(\bar{\Omega})$. We can improve 
the H\"older regularity of the solution if certain combinations of principal curvatures of the boundary 
do not vanish and thus establish a relation between the H\"older exponents for $f$ 
and principal curvatures of the boundary. We will also discuss the global regularity of the solution 
$f$ for certain domains with singularity. 

The first main result is given by the following theorem. 

\begin{theorem}\label{thrm-main}
Assume that $\Omega\subset \mathbb{R}^{n}$ is a bounded $C^2$-domain with $H_{\partial\Omega}\geq 0$ 
and that $f\in C(\bar\Omega)\cap C^\infty (\Omega)$ is the solution of \eqref{eq-main}. Then, 
$f\in  C^{\frac{1}{n+1}}(\bar{\Omega})$.  Moreover, 
$$[f]_{C^{\frac{1}{n+1}}(\bar{\Omega})} \leq [(n+1) \operatorname{diam}(\Omega)^{n}]^{\frac{1}{n+1}},$$ 
where $\operatorname{diam}(\Omega)$ is the diameter of $\Omega$.
\end{theorem}

We point out that the H\"{o}lder exponent $\frac{1}{n+1}$ is optimal. 
By Remark \ref{remark} below, we cannot improve the  regularity for $f$ 
in domains with nonnegative mean curvature in general. 
We also note that $n+1$ is the power of the first global term in the expansions of minimal graphs 
in the hyperbolic space.  See \cite{HanJiang2014}. However, if certain combinations of principal curvatures of the boundary 
do not vanish, then we can improve the global H\"older regularity. 
We will prove  in 
Theorem \ref{thrm-local-regularity} that solutions can be $C^{1/i}$ up to the boundary 
under appropriate conditions, for an even integer $2\le i\le n$. 
This confirms what Lin suggested in 
Remark 3.7 \cite{Lin1989Invent}.

The proof of Theorem \ref{thrm-main} is based on the maximum principle and 
the recent work of Han and Jiang \cite{HanJiang2014} on the boundary expansions for minimal graphs 
in the hyperbolic space.

We note that the estimate in Theorem \ref{thrm-main} does not depend on the regularity 
of the domain. This allows us to discuss \eqref{eq-main} in domains with singularity. 
Along this direction, we prove the following result. 

\begin{theorem}\label{thrm-reg-peiece-mean-conv-dom}
Let $\Omega\subset \mathbb{R}^{n}$ be a bounded convex domain 
which is the intersection of finitely many bounded convex $C^{2}$-domains $\Omega_i$ 
with $H_{\partial\Omega_i} > 0$, and  let 
$f\in C(\bar\Omega)\cap C^\infty (\Omega)$ be the solution of \eqref{eq-main}. 
Then $f\in C^{{1}/{2}}(\bar\Omega)$, and
$$[f]_{C^{{1}/{2}}(\bar{\Omega})} \leq C,$$
where $C$ is a positive constant depending only on $n$, 
$H_{\partial \Omega_i}$ and the diameter of $\Omega_i$.
\end{theorem}

The equation in \eqref{eq-main} for $n=2$ also appears in the study of the Chaplygin gas. See 
\cite{Serre2009} for details. 

The paper is organized as follows. In Section \ref{sec-mean-convex}, we discuss 
\eqref{eq-main} in domains with nonnegative boundary mean curvature and prove 
Theorem \ref{thrm-main}. In Section \ref{sec-convex}, we discuss 
\eqref{eq-main} in convex domains and prove 
Theorem \ref{thrm-reg-peiece-mean-conv-dom}. In Section \ref{sec-EquivalentEquation}, 
we discuss the regularity of solutions of an equivalent form of the equation in \eqref{eq-main}, 
which appears in the study of the Chaplygin gas.

We would like to thank Xumin Jiang for many helpful comments and suggestions.

\section{General Mean Convex Domains}\label{sec-mean-convex}

We first note that, for $\Omega =B_{R}(x_0) =\{ x \in \mathbb{R}^{n}:|x-x_0|<R\}$, 
the unique solution of \eqref{eq-main} is given by
$$f_R(x) = (R^2 -|x-x_0|^2)^{\frac{1}{2}}.$$
Hence, for a domain $\Omega$ with $H_{\partial\Omega}\ge 0$, we have, by the maximum principle, 
$$ |f|_{L^{\infty}(\Omega)}\le \operatorname{diam}(\Omega).$$
We also note that the gradient of $f$ blows up near $\partial \Omega$.

Now we prove two lemmas. Throughout the proof, we denote by $d(x)$ 
the distance from $ x$ to $\partial \Omega$ and by $\Lambda$ 
the maximum of the absolute value of principle curvatures of $\partial\Omega$.

\begin{lemma}\label{lemma-f-order}
Assume that $\Omega\subset \mathbb{R}^{n}$ is a bounded $C^2$-domain with $H_{\partial\Omega}\geq 0$ 
and that $f\in C(\bar\Omega)\cap C^\infty (\Omega)$ is the solution of \eqref{eq-main}. 
Then,
$$f\leq Cd^{\frac{1}{n+1}}\quad\text{in }\Omega,$$ 
where $C$ is a positive constant depending only on $n$, $\Lambda$ and the diameter of $\Omega$.
\end{lemma}

\begin{proof}
Set $w=\psi(d),$ for some function $\psi$ to be determined.  We will require $\psi>0$ and $\psi'>0$ 
on $(0,\delta)$, for some $\delta>0$. Then, 
\begin{align*}
w_{i} & =\psi'd_{i} ,\\
w_{ij}& =\psi'd_{ij}+\psi''d_{i}d_{j},\end{align*}
and hence, 
$$w_{i}w_{j}w_{ij} =\psi'^{2}d_{i}d_{j}(\psi'd_{ij}+\psi''d_{i}d_{j})=\psi'^{2}\psi'',$$
by $d_{i}d_{ij}=0.$ Therefore, 
\begin{align*}
&\Delta w-\frac{w_{i}w_{j}}{1+|\nabla w|^{2}}w_{ij}+\frac{n}{w}\\
&=\psi'\Delta d+\psi''-\frac{1}{1+\psi'^{2}}\psi''\psi'^{2}+\frac{n}{\psi}\\
&=\frac{1}{1+\psi'^{2}}\psi''+\frac{n}{\psi}+\psi'\Delta d\\
&\leq \frac{1}{1+\psi'^{2}}\psi''+\frac{n}{\psi},
\end{align*} where we used the assumption $H_{\partial\Omega}\geq 0 $ 
and the expansion of $\Delta d$ as in \cite{GT1983}. Set 
$$m(\psi)=\frac{\psi''\psi}{1+\psi'^{2}}+n.$$
Then, 
$$\Delta w-\frac{w_{i}w_{j}}{1+|\nabla w|^{2}}w_{ij}+\frac{n}{w}
\le \frac{1}{\psi}m(\psi).$$

In the following, we set 
$$\psi(d)=A[d^{p}-d^{q}],$$
for some positive constants $A$, $p$ and $q$ to be determined. Then,
\begin{align*}
\psi'& =A[pd^{p-1}-qd^{q-1}],\\
\psi''& =A[p(p-1)d^{p-2}-q(q-1)d^{q-2}].
\end{align*}
Hence, 
$$1+\psi'^{2}=1+A^{2}[p^{2}d^{2p-2}-2pqd^{q+p-2}+q^{2}d^{2q-2}],$$ 
and 
\begin{align*}
\psi''\psi+n(1+\psi'^{2})&=
n+A^{2}\big[p(p-1+pn)d^{2p-2}\\
&\qquad-(p^{2}+q^{2}-p-q+2pqn)d^{q+p-2}+q(q-1+nq)d^{2q-2}\big].\end{align*}
In the following, we set 
\begin{equation}\label{eq-expression-p}p=\frac{1}{n+1}.\end{equation}
Then, $p-1+pn=0$ and hence 
\begin{align*}
\psi''\psi+n(1+\psi'^{2})=
n+A^{2}\big[-h(q)d^{q+p-2}+q(q-1+nq)d^{2q-2}\big],\end{align*}
where 
$$h(q)=p^{2}+q^{2}-p-q+2pqn.$$
Then, 
$$m(\psi)=\frac{n+A^{2}\big[-h(q)d^{q+p-2}+q(q-1+nq)d^{2q-2}\big]}
{1+A^{2}[p^{2}d^{2p-2}-2pqd^{q+p-2}+q^{2}d^{2q-2}]}.$$
Note that $h(q)$ is a quadratic polynomial of $q$. 
With the expression of $p$ in \eqref{eq-expression-p}, we have 
\begin{align*}h(q)&=q^2+(2pn-1)q+p^2-p=q^2+(np-p)q-np^2\\
&=(q-p)(q+np).\end{align*}
Then, $h(q)>0$ for $q>p$. A simple rearrangement yields 
$$m(\psi)=d^{q-p}\cdot 
\frac{-A^{2}(q-p)(q+np)+nd^{2-p-q}+A^2q(q-1+nq)d^{q-p}}
{A^{2}p^{2}+d^{2-2p}-2A^2pqd^{q-p}+A^2q^{2}d^{2q-2p}}.$$
Now, we take $q$ such that 
$$p<q<2-p.$$
Then, we can take $\delta$ small depending only on $n$ and $\Lambda$ such that 
$m(\psi)\le 0$ for $d\in (0,\delta)$. 
Next, we choose $A$ large, depending only on $n$, $\Lambda$ and the 
$L^\infty(\Omega)$-norm of $f$, such that 
$$|f|_{L^{\infty}(\Omega)}\le A(\delta^{p}-\delta^{q})=\psi(\delta).$$
Hence, by the maximum principle, we have 
$f\leq \psi$ for $0<d<\delta$, 
and therefore 
$$f\leq Ad^{\frac{1}{n+1}}\quad\text{in }\{x\in\Omega:\, d(x)\leq \delta\}.$$
This implies the desired result. \end{proof}

Next, we proceed as Lin \cite{Lin1989Invent}. 
Locally near each boundary point, the graph of $f$ can be represented by a function 
over its vertical tangent plane. Specifically, 
we fix a boundary point of $\Omega$, say the origin, and assume that 
the vector $e_n=(0,\cdots, 0,1)$ is the interior normal vector to $\partial\Omega$ 
at the origin. Then, with $x=(x',x_n)$, the $x'$-hyperplane is the tangent plane of 
$\partial\Omega$ at the origin, and the boundary $\partial\Omega$ can be expressed 
in a neighborhood of the origin as a graph of a smooth function over $\mathbb R^{n-1}\times\{0\}$, 
say 
$$x_n=\varphi(x').$$
We now denote points in $\mathbb R^{n+1}=
\mathbb R^n\times\mathbb R$ by $(x',x_n,t)$. The vertical hyperplane 
given by $x_n=0$ is the tangent plane to the graph of $f$ at the origin in $\mathbb R^{n+1}$,
and we can represent the graph of $f$ as a graph of a new function $u$ defined in terms of 
$(x', 0, t)$ for small $x'$ and $t$, with $t>0$. In other words, we treat 
$\mathbb R^n=\mathbb R^{n-1}\times\{0\}\times\mathbb R$ as our new base space and write 
$u=u(x', t)$. Then, for some $R>0$,
$u$ satisfies 
\begin{align}\label{eq-Intro-Equ}
\Delta u - \frac{u_i u_j}{1+|\nabla u|^2}u_{ij}-\frac{n u_{t}}{t}=0  \quad \text{in } B_R^+,
\end{align}
and 
\begin{align}\label{eq-Intro-EquCondition}
u(\cdot, 0)=\varphi\quad\text{on }B_R'.
\end{align}
We note that $u$ and $f$ are related by 
\begin{align}\label{eq-u}
x_{n}=u(x',t),
\end{align}
and 
\begin{align}\label{eq-f}
t=f(x',x_n).
\end{align}
Set
\begin{align}\label{b1a-v}
u_{n+1}(x',t)=\varphi(x')+\sum_{i=2}^{n+1}c_i(x')t^i+c_{n+1,1}(x')t^{n+1}\log t.
\end{align}
In fact, $c_i=0$ for odd $i$ between $2$ and $n$
and $c_{n+1,1}=0$ for even $n$. We have the following result. 

\begin{lemma}\label{lemma-Expension-u}
For some constant $\alpha\in (0,1)$, 
let $\varphi\in C^{n+1,\alpha}(B'_R)$ be a given function 
and $u\in C(\bar B^+_R)\cap C^\infty(B^+_R)$ be 
a solution of \eqref{eq-Intro-Equ}-\eqref{eq-Intro-EquCondition}. Then, 
there exist functions  $c_i\in C^{n+1-i, \epsilon}(B_R')$, for $i=0, 2, 4, \cdots, n+1$, 
$c_{n+1,1}\in C^{\epsilon}(B_R')$, 
and any $\epsilon\in (0,\alpha)$, 
such that, 
for $u_{n+1}$ defined as in \eqref{b1a-v},
for any $m=0, 1, \cdots, n+1$, any $\epsilon\in (0,\alpha)$, and any $r\in (0, R)$, 
\begin{equation}\label{eq-MainRegularity}\partial_t^m (u-u_{n+1})\in C^{\epsilon}(\bar B^+_r),
\end{equation}
and, for any $(x',t)\in 
B^+_{R/2}$,  
\begin{equation}\label{eq-MainEstimate}|\partial_t^m (u-u_{n+1})(x',t)|
\le C t^{n+1-m+\alpha},
\end{equation}
for some positive constant $C$ depending only on $n$,  $\alpha$, $R$, 
the $L^\infty$-norm of $u$ in $B_R^+$ and the $C^{n+1, \alpha}$-norm  of $\varphi$ in 
$B_R'$.  
\end{lemma} 

Lemma \ref{lemma-Expension-u} follows from Theorem 1.1 in \cite{HanJiang2014}
by taking $\ell=k=n+1$. In fact, $c_2, \cdots, c_n$ and $c_{n+1,1}$ are coefficients 
for local terms and have explicit expressions in terms of $\varphi$. Meanwhile, 
$c_{n+1}$ is the coefficient of the first nonlocal term. 

\begin{remark}\label{remark}
The growth rate $d^{\frac{1}{n+1}}$ in Lemma \ref{lemma-f-order}
is optimal for $f$ in general domains with nonnegative boundary mean curvature. This can be seen 
as follows. Consider a $C^2$-domain $\Omega$ with $H_{\partial\Omega}\ge 0$ such that 
$B_{r}\bigcap\partial\Omega\subset \{x_n=0\}$
and $B_{r}\bigcap\Omega= B_r^+$, for some $r>0$. Then, $\varphi=0$ on 
$B'_r$. Hence, in the expression of $u_{n+1}$, 
we have $c_{i}=0$ for $i\leq n$ and $c_{n+1,1}=0.$ The estimate \eqref{eq-MainEstimate} with $m=0$
implies $|u|\leq Ct^{n+1}$ near the origin. Moreover, 
$x_n=u\ge 0$ in $B_{r}\bigcap\Omega$. Hence,  
$C^{-1}d^{\frac{1}{n+1}}\leq f$ near the origin. By combining with Lemma \ref{lemma-f-order}, 
we obtain 
$$C^{-1}d^{\frac{1}{n+1}}\le f \leq Cd^{\frac{1}{n+1}}\quad\text{near the origin}.$$
Therefore, the growth rate of $f$ is exactly $d^{\frac{1}{n+1}}$.  
\end{remark}

The next result plays a crucial role in the proof of Theorem \ref{thrm-main}. 

\begin{lemma}\label{lemma-Df=0} 
Assume that $\Omega\subset \mathbb{R}^{n}$ is a bounded $C^2$-domain 
with $H_{\partial\Omega}\geq 0$ 
and that $f\in C(\bar\Omega)\cap C^\infty (\Omega)$ is the solution of \eqref{eq-main}. 
Then, 
$$f^{n}\sqrt{1+|\nabla f|^{2}}\leq|f|_{L^{\infty}(\Omega)}^{n}\quad\text{in }\Omega.$$
\end{lemma}

\begin{proof} We will prove, 
for any $\varepsilon>0,$ 
\begin{equation}\label{eq-estimate-epsilon}
f^{n+\varepsilon}\sqrt{1+|\nabla f|^{2}}\le |f|_{L^{\infty}(\Omega)}^{n+\varepsilon}
\quad\text{in }\Omega.\end{equation}
By letting $\varepsilon\to 0$, we have the desired result. 

We first consider the case that $\partial\Omega$ is smooth and set
$$F_{\varepsilon}=f^{n+\varepsilon}\sqrt{1+|\nabla f|^{2}}.$$
The proof consists of two steps.

{\it Step 1}. We will prove $F_{\varepsilon}\to 0$ as $x$ approaches $\partial\Omega.$  
Take any $x\in\Omega$.  Without loss of generality, take a coordinate 
such that $x=(0,x_{n})$, with $ x_{n}=d(x)$, and 
the origin is the nearest point on $\partial\Omega$ to $x.$ 
We can express $x_{n}$ by a function $x_n=u(x',t)$
as in \eqref{eq-u}, which satisfies \eqref{eq-Intro-Equ}. 

By Lemma \ref{lemma-f-order}, we have, for $x_n\in (0,\delta)$,  
$$f(0,x_n)\leq Ax_n^{\frac{1}{n+1}}.$$ 
Hence, 
\begin{equation}\label{eq-lower-bound}u(0,t)\geq Ct^{n+1},\end{equation} 
for some positive constant $C$. Therefore,
$x_{n}$ approaching $0$ is equivalent to 
$t$ approaching $0$. 

Next, we note $\varphi(0)=0$ since $\{x_n=0\}$ is tangent to $\partial\Omega$ 
at the origin. By \eqref{eq-lower-bound} and \eqref{eq-MainEstimate} with $m=0$, 
there is a nonzero term in the expression of $u_{n+1}(0,t)$ in \eqref{b1a-v}. We now write, 
for some $\alpha\in (0,1)$, 
$$u(0,t)=\sum_{i=2}^{n}c_i(0)t^i+c_{n+1,1}(0)t^{n+1}\log t+c_{n+1}(0)t^{n+1}+O(t^{n+1+\alpha}).$$
By $u(0,t)>0$ for small $t>0$, we note that either the first nonzero coefficient $c_i(0)$ is positive, for some 
$i=0, 1, \cdots, n+1$, or  $c_{n+1,1}(0)<0$ if $c_i(0)=0$ for any $i=2, \cdots, n$. 
Next, 
$$\aligned tu_t(0,t)&=\sum_{i=2}^{n}ic_i(0)t^i+(n+1)c_{n+1,1}(0)t^{n+1}\log t\\
&\qquad+\big[c_{n+1,1}(0)+(n+1)c_{n+1}(0)\big]t^{n+1}+O(t^{n+1+\alpha}).\endaligned$$
Therefore, we have, for $t$ small, 
$$\frac{tu_{t}(0,t)}{u(0,t)}>1.$$ 
By \eqref{eq-u} and \eqref{eq-f}, we get 
\begin{align*}
1 =u_{t}f_{x_{n}},\quad
0 =u_{x'}+u_{t}f_{x'}.
\end{align*} 
Hence, 
$$|u_t\nabla_xf|^2=1+|\nabla_{x'}u|^2.$$
As a result, we obtain
\begin{align*}
F_{\varepsilon}\leq f^{n+\varepsilon}+f^{n+\varepsilon}\frac{\sqrt{n}}{u_{t}}
\leq f^{n+\varepsilon}+C(n)f^{\varepsilon},
\end{align*}
where we used 
\begin{align*}
f^{n+\varepsilon}\frac{1}{u_{t}}<\frac{f^{n+\varepsilon}t}{u}=\frac{f^{n+1}}{d}f^{\varepsilon}
\leq Cf^{\varepsilon}.
\end{align*}
Hence, $F_{\varepsilon}\to0,$ as $x_{n}$ or $t$ approaches $0.$
We point out that it is important to have the extra power of $\varepsilon$. 

{\it Step 2}. By Step 1, we note that $F_{\varepsilon}$ attains its maximum at some 
$x_{0}$ in $\Omega.$ We will prove $\nabla f(x_{0})=0$ by contradiction. 
Without loss of generality, we assume $|\nabla f(x_{0})|=f_{1}(x_{0})\neq 0.$ 
Set 
$$g_\varepsilon=\log F_\varepsilon=\log( f^{n+\varepsilon}\sqrt{1+|\nabla f|^{2}}).$$ 
Then, $g_{\varepsilon}$ attains its maximum at $x_{0}.$ Hence, 
$g_{\varepsilon,i}(x_0)=0$ and $(g_{\varepsilon, ij}(x_0))\le 0$. 
A simple calculation yields 
\begin{align*}
g_{\varepsilon, i} = (n+\varepsilon)\frac{f_{i}}{f}+\frac{f_{k}f_{ki}}{1+|\nabla f|^{2}},
\end{align*}
and
\begin{align*}
g_{\varepsilon,ij}=(n+\varepsilon)\frac{f_{ij}}{f}-(n+\varepsilon)\frac{f_{i}f_j}{f^{2}}
+\frac{f_{ki}f_{kj}}{1+|\nabla f|^{2}}
+\frac{f_{kij}f_{k}}{1+|\nabla f|^{2}}-\frac{2f_kf_lf_{ki}f_{lj}}{(1+|\nabla f|^{2})^{2}}.
\end{align*}
In the following, we calculate at the point $x_{0}.$ 
By $g_{\varepsilon, i}=0$, $f_2=\cdots f_n=0$ and $f_1\neq 0$, we have 
$$\aligned f_{11}&=-\frac{n+\varepsilon}{f}(1+f_1^2),\\
f_{1i}&=0\, (i\neq 1).\endaligned 
$$

Set 
$$a_{ij}(p)=\delta_{ij}-\frac{p_{i}p_{j}}{1+|p|^{2}}.$$
Then, 
$$a_{ij}(\nabla f)f_{ij}=-\frac{n}{f}.$$ 
A simple differentiation yields 
\begin{align}\label{eq-aij}
a_{ij}f_{1ij}+a_{ij,p_{k}}f_{ij}f_{1k}=\frac{nf_{1}}{f^{2}}.
\end{align} 
By the assumption $|\nabla f|=f_{1}$ at $x_0$, 
we have 
$$a_{11}=\frac{1}{1+f_{1}^{2}},\quad 
a_{ii}=1\,(i\neq 1),\quad 
a_{ij}=0\,(i\neq j).$$
Moreover, a straightforward calculation yields 
$$\aligned 
&a_{11,p_{1}}=-\frac{2f_{1}}{(1+f_{1}^{2})^{2}},\quad 
a_{11,p_{i}}=0\, (i\neq 1),\\
&a_{1i,p_{i}}=-\frac{f_{1}}{1+f_{1}^{2}}\,(i\neq1),\quad
a_{1i,p_{j}}=0\,(i\neq 1,i\neq j),\\
&a_{ij,p_{k}}=0\,(i\neq 1,j\neq 1).\endaligned$$
Then, we can rewrite \eqref{eq-aij} as 
\begin{align*}
a_{ii}f_{1ii}+a_{11,p_{1}}f_{11}^{2}+2\sum_{i\geq 2}a_{1i,p_{i}}f_{1i}^{2}=\frac{nf_{1}}{f^{2}},
\end{align*}
or 
$$a_{ii}f_{1ii}=\frac{2f_{1}}{1+f_{1}^{2}}\sum_{i\geq 1}a_{ii}f_{1i}^{2}+  \frac{nf_{1}}{f^{2}}.$$
If $f_{1}\neq 0, $ we have, by a simple substitution,
$$a_{ij}g_{\varepsilon,ij}=-\frac{n(n+\varepsilon)}{f^2}-(n+\varepsilon)\frac{a_{11}f_1^2}{f^{2}}
+\frac{a_{ii}f^2_{ki}}{1+f_1^{2}}
+\frac{a_{ii}f_{1ii}f_{1}}{1+f_1^{2}}-\frac{2a_{ii}f_1^2f_{1i}^2}{(1+f_1^{2})^{2}},
$$
and, keeping only $k=1$ in the middle term, 
$$\aligned
a_{ij}g_{\varepsilon,ij}&\ge-\frac{n(n+\varepsilon)}{f^2}-(n+\varepsilon)\frac{a_{11}f_1^2}{f^{2}}
+\frac{a_{11}f^2_{11}}{1+f_1^{2}}\\
&\qquad+\frac{f_{1}}{1+f_1^{2}}\left(\frac{2f_{1}}{1+f_{1}^{2}}a_{11}f_{11}^{2}+  n\frac{f_{1}}{f^{2}}\right)
-\frac{2a_{11}f_1^2f_{11}^2}{(1+f_1^{2})^{2}}.\endaligned
$$
Hence, 
\begin{align*}
a_{ij}g_{\varepsilon, ij}
&\ge \frac{\varepsilon}{f^{2}}\left(n+\varepsilon-\frac{f_{1}^{2}}{1+f_{1}^{2}}\right)>0.
\end{align*}
On the other hand, since $g_{\varepsilon}$ attains its maximum at $x_{0},$ 
we have $a_{ij}g_{\varepsilon, ij}\leq 0$, which leads to a contradiction. 
Therefore, $\nabla f(x_{0})=0$, and by the definition of $F_{\varepsilon},$ we have
\begin{align*}
F_{\varepsilon}\leq |f|_{L^{\infty}(\Omega)}^{n+\varepsilon}.
\end{align*}
This implies \eqref{eq-estimate-epsilon} in the case that $\partial\Omega$ is smooth. 

We now consider the general case that $\partial\Omega$ is $C^{2}$ 
with $H_{\partial\Omega}\ge 0$. We can take a sequence of smooth domains 
$\{\Omega_{k}\}$ with $H_{\partial\Omega_{k}}\geq 0$ such that 
$\partial\Omega_k$ approaches $\partial\Omega$ in $C^2$. 
Let $f_k\in C(\bar\Omega_k)\cap C^\infty(\Omega_k)$ be the solution of
\begin{align*}
\Delta f_{k}-\frac{f_{k,i}f_{k,j}}{1+|\nabla f_{k}|^{2}}f_{k,ij}+\frac{n}{f_{k}}&=0 \quad\text{in }\Omega_{k},\\
f_{k}&=0 \quad \text{on } \partial\Omega_{k},\\
f_{k}&>0 \quad\text{in }  \Omega_{k}.
\end{align*}
By what we just proved, we have \begin{align*}
f_{k}^{n+\epsilon}\sqrt{1+|\nabla f_{k}|^{2}}\leq |f_{k}|_{L^{\infty}(\Omega_{k})}^{n+\epsilon}.
\end{align*}
By the interior estimate and Lemma \ref{lemma-f-order}, we have $f_{k}(x)\to f(x)$
and $\nabla f_{k}(x)\to \nabla f (x)$ for any $x \in \Omega$. Hence, by taking the
limit, we obtain
\begin{align*}
f^{n+\epsilon}\sqrt{1+|\nabla f|^{2}}\leq |f|_{L^{\infty}(\Omega)}^{n+\epsilon}.
\end{align*}
This is \eqref{eq-estimate-epsilon} in the general case. 
\end{proof}

Now, we are ready to prove Theorem \ref{thrm-main}.

\begin{proof}[Proof of Theorem \ref{thrm-main}]
By Lemma \ref{lemma-Df=0}, we have
\begin{align*}
f^{n}\sqrt{1+|\nabla f|^{2}}\leq |f|_{L^{\infty}(\Omega)}^{n}.
\end{align*}
Hence,
\begin{align}\label{eq-derivative-(n+1)}
|\nabla f^{n+1}|< (n+1)\operatorname{diam}(\Omega)^{n}.
\end{align}
By the mean value theorem, we obtain, for any $x\in\Omega$, 
$$f^{n+1}(x)=|f^{n+1}(x)-0| \leq  |\nabla f^{n+1}| d(x),$$
and hence
\begin{equation}\label{eq-growth-rate}
f(x)\leq \big[(n+1) \operatorname{diam}(\Omega)^{n}\big]^{\frac{1}{n+1}} d(x)^{\frac{1}{n+1}}.
\end{equation}
We point out that \eqref{eq-growth-rate} is sharper than Lemma \ref{lemma-f-order}.  

Next, we note, for any $x_1, x_2\in \Omega$,  
\begin{align}\label{eq-n+1}
|f(x_{1})-f(x_{2})|^{n+1}\leq |f(x_{1})^{n+1}-f(x_{2})^{n+1}|.
\end{align}
In fact, for $f(x_{1}),f(x_{2})>0,$ we may assume 
$f(x_{1})=\max\{f(x_{1}),f(x_{2})\},$ and then employ  $|1-y|^{n+1}\leq |1-y^{n+1}|,$ 
with $y=\frac{f(x_{2})}{f(x_{1})},$ to derive (\ref{eq-n+1}). 

Now, we claim, for any $x_1, x_2\in \Omega$, 
\begin{equation}\label{eq-Holder}
|f(x_1)-f(x_2)|\leq \big[(n+1) \operatorname{diam}(\Omega)^{n}\big]^{\frac{1}{n+1}}
|x_1-x_2|^{\frac{1}{n+1}}.\end{equation}
With $d_i=\operatorname{dist}(x_i, \partial\Omega)$ for $i=1,2$, we assume $d_{1}\ge d_2.$

If $|x_{1}-x_{2}|\geq d_{1},$ by \eqref{eq-growth-rate}, we have 
$$\aligned 
|f(x_{1})-f(x_{2})|&\leq  |f(x_1)|\le
\big[(n+1) \operatorname{diam}(\Omega)^{n}\big]^{\frac{1}{n+1}} d_{1}^{\frac{1}{n+1}}\\
&\leq \big[(n+1)\operatorname{diam}(\Omega)^{n}\big]^{\frac{1}{n+1}} |x_{1}-x_{2}|^{\frac{1}{n+1}}.
\endaligned$$
If $|x_{1}-x_{2}|< d_{1},$ by \eqref{eq-n+1} and \eqref{eq-derivative-(n+1)}, we have,  
\begin{align*}
|f(x_{1})-f(x_{2})|^{n+1}&\leq |f(x_{1})^{n+1}-f(x_{2})^{n+1}|\le |\nabla f(\widetilde x)^{n+1}||x_1-x_2|\\
&\leq  (n+1) \operatorname{diam}(\Omega)^{n}|x_{1}-x_{2}|,
\end{align*}
where $\widetilde x$ is some point in $B_{d_1}(x_1)\subset \Omega$. 
In summary, we have \eqref{eq-Holder}. 
\end{proof}

\section{Convex Domains}\label{sec-convex}

In this section, we discuss refined regularity for $f$ in general convex domains.
First, we prove that  \eqref{eq-main} admits a solution in convex domains. We point out 
that there is no higher regularity assumptions on the boundary of the domains. 

\begin{theorem}\label{thrm-ConvexDom}
Assume that $\Omega\subset \mathbb{R}^{n}$ is a bounded convex domain. 
Then, \eqref{eq-main} admits a unique solution
$f\in C(\bar{\Omega})\cap C^{\infty}(\Omega)$ and $f$ is concave.  Moreover,
$f\in C^{\frac{1}{n+1}}(\bar{\Omega})$ and 
$$[f]_{C^{\frac{1}{n+1}}(\bar{\Omega})} \leq [(n+1) \operatorname{diam}(\Omega)^{n}]^{\frac{1}{n+1}},$$
where $\operatorname{diam}(\Omega)$ is the diameter of $\Omega$.
\end{theorem}

\begin{proof}
We first prove the existence and note that the uniqueness is a simple 
consequence of the maximum principle.

We take a sequence of bounded smooth  convex domains
$\{\Omega_{k}\}$  such that
$\partial\Omega_k$ approaches $\partial\Omega$ in the Hausdorff metric.
Let $f_k\in C(\bar\Omega_k)\cap C^\infty(\Omega_k)$ be the solution of
\begin{align*}
\Delta f_{k}-\frac{f_{k,i}f_{k,j}}{1+|\nabla f_{k}|^{2}}f_{k,ij}+\frac{n}{f_{k}}&=0 \quad\text{in }\Omega_{k},\\
f_{k}&=0 \quad \text{on } \partial\Omega_{k},\\
f_{k}&>0 \quad\text{in }  \Omega_{k}.
\end{align*}
By \eqref{eq-growth-rate} and  the interior estimate, we have, 
for any $m\ge 1$ and any $\Omega'\subset\subset\Omega$,
$$f_ k \rightarrow f\quad\text{in }C^m(\Omega'),$$
for some function $f\in C(\bar \Omega)\cap C^\infty(\Omega)$ with 
$f =0$ on $\partial\Omega$. Therefore, $f$ is the unique solution of \eqref{eq-main}.
Next, we apply 
Theorems 3.1 and 3.2 \cite{Kennington1985} in $\Omega_k$ and conclude that 
$f_k$ are concave. Hence, 
$f$ is  concave  in $\Omega$.

For the global regularity, we take  any $x_1, x_2\in \Omega$. Then, 
$x_1, x_2\in \Omega_k$ for $k$ large, and hence
\begin{equation*}
|f_{k}(x_1)-f_k(x_2)|\leq \big[(n+1) \operatorname{diam}(\Omega_k)^{n}\big]^{\frac{1}{n+1}}
|x_1-x_2|^{\frac{1}{n+1}}.\end{equation*}
By letting $k\to \infty$, we get 
\begin{equation*}
|f(x_1)-f(x_2)|\leq \big[(n+1) \operatorname{diam}(\Omega)^{n}\big]^{\frac{1}{n+1}}
|x_1-x_2|^{\frac{1}{n+1}}.\end{equation*}
This implies the desired result on the H\'older semi-norm of $f$. 
\end{proof}

The H\"{o}lder exponent $\frac{1}{n+1}$ is optimal.
By Remark \ref{remark}, we cannot improve the regularity for $f$
in general convex domains.

We next consider the local regularity for $f$. 
We write the equation \eqref{eq-main} in its divergence form
$$\nabla\frac{\nabla f}{\sqrt{1+ |\nabla f|^{2}}} +\frac{n}{f \sqrt{1+ |\nabla f|^{2}}}=0.$$
Then,  we have, for any $\varphi \in C^{\infty}_{0}(\Omega)$,
\begin{equation}\label{eq-interg-by-part}
\int_{\Omega} \frac{\nabla f \nabla\varphi}{\sqrt{1+ |\nabla f|^{2}}} 
-\int_{\Omega} \frac{n \varphi}{f \sqrt{1+ |\nabla f|^{2}}}=0.
\end{equation}

We now prove a local regularity for solutions of \eqref{eq-main}. We point out that there is no 
regularity assumption on the domain. Hence, it may be applied to domains with singularity. 

\begin{theorem}\label{thrm-regl-growth-rate}
Let $\Omega\subset \mathbb{R}^{n}$ be a bounded domain and
$f\in C(\bar\Omega)\cap C^\infty (\Omega)$ be a solution of \eqref{eq-main}. 
For some $x_0\in\partial\Omega$ and $\alpha\in (0,1)$, assume
$$f\leq Md^{\alpha},\quad \Delta f \le0\quad\text{in }\Omega\cap B_r(x_0),$$
for some constants $M\ge 1$ and $r>0$. 
Then, $f\in C^{\alpha}(\bar\Omega\cap B_{r/2}(x_0))$, and 
\begin{equation}\label{eq-HolderEstimate}[f]_{C^{\alpha}(\bar\Omega\cap B_{r/2}(x_0))}\le C M,\end{equation}
where $C$ is a positive constant depending only on $n$ and $\alpha$. If, in addition, 
$f$ is concave in $\Omega\cap B_r(x_0)$ and $\alpha\in (0,1/2]$, then 
\begin{align}\label{eq-GradientEstimate}
|\nabla f^{\frac{1}{\alpha}}|_{L^\infty(\Omega\cap B_{r/2}(x_0))}\leq CM^{\frac1\alpha},
\end{align}
where $C$ depends on $n$ and $\alpha$. 
\end{theorem}

\begin{proof}
{\it Step 1.} We will prove, for any $x\in\Omega\cap B_{r/2}(x_0)$, 
\begin{equation}\label{f-Lp-int}
[f]_{C^{\alpha}(B_{d_{x}/2}(x))}\le C M,\end{equation}
where $C$ is a positive constant depending only on $n$ and $\alpha$.
Then, we have \eqref{eq-HolderEstimate} by combining with 
$f\leq Md^{\alpha}$ in $\Omega\cap B_r(x_0)$. Here and hereafter, 
we write $d_x=d(x)$. 

We now fix a point $x\in\Omega\cap B_{r/2}(x_0)$ and a cutoff function 
$\psi\in C_0^\infty(B_{d_x}(x))$, with $0\le\psi\le 1$ and $|\nabla \psi|\le Cd_{x}^{-1}$, 
for some positive constant $C$ depending only on $n$. In addition, we assume 
$\psi=1$ in $B_{d_x/2}(x)$. 

We note $d\le 2d_x$ in $B_{d_x}(x)$. To verify this, we take 
$\widetilde x\in \partial B_{d_x}(x)\cap \partial\Omega$. Then, for any $y\in B_{d_x}(x)$, 
$d(y)\le |y-\widetilde x|\le |y-x|+|x-\widetilde x|\le 2d_x$. This implies 
$$f\le 2Md_x^\alpha\quad\text{in }B_{d_x}(x).$$
Next, we note $f \geq d$. This follows from a simple comparison of $f$ 
and the corresponding solution in $B_{d(y)}(y)$, for any $y\in\Omega$. 
Then, for any $y\in B_{d_x}(x)$, we have, for some $y_0\in\partial B_{d_x}(x)$,  
$$f(y)\ge d(y)\ge |y-y_0|,$$
and, for some $\widetilde y$ between $y$ and $y_0$,  
$$\psi(y)=\psi(y)-\psi(y_0)\le |\nabla \psi(\widetilde y)||y-y_0|\le Cd_x^{-1}|y-y_0|.$$
Hence, 
\begin{equation}\label{eq-ratio}\frac\psi f\le Cd_x^{-1}\quad\text{in }B_d(x).\end{equation}

Taking $ \varphi=f \psi$ in \eqref{eq-interg-by-part},  we have
$$\int\psi\sqrt{1+|\nabla f|^2}=\int\frac{(n+1)\psi}{\sqrt{1+|\nabla f|^2}}
-\frac{f\nabla f\nabla\psi}{\sqrt{1+|\nabla f|^2}}\le (n+1)\int\psi+\int f|\nabla\psi|.$$ 
Then, 
\begin{equation}\label{eq-iteration1}\int\psi\sqrt{1+|\nabla f|^2}\le C M d_x^{\alpha -1+ n}.
\end{equation}

For an integer $k\ge 2$, take $\varphi= \psi^{k}(1+ |\nabla f|^{2})^{\frac{k-1}{2}}$ 
in \eqref{eq-interg-by-part}. Then, 
\begin{align*}
&\int k\psi^{k-1}(1+ |\nabla f|^{2})^{\frac{k-2}{2}} \nabla f\nabla\psi  
-\int\frac{n}{f}\psi^{k}(1+ |\nabla f|^{2})^{\frac{k-2}{2}}\\
&\qquad+(k-1)\int\psi^{k}(1+ |\nabla f|^{2})^{\frac{k-2}{2}}\frac{f_{i}f_{j}}{1+|\nabla f|^{2}}f_{ij}=0.
\end{align*}
Note 
$$\frac{f_{i}f_{j}}{1+|\nabla f|^{2}}f_{ij}=\Delta f+\frac{n}{f}.$$
A simple substitution yields 
\begin{align*}
&\int k\psi^{k-1}(1+ |\nabla f|^{2})^{\frac{k-2}{2}} \nabla f\nabla\psi  
+n(k-2)\int\frac{1}{f}\psi^{k}(1+ |\nabla f|^{2})^{\frac{k-2}{2}}\\
&\qquad+(k-1)\int\psi^{k}(1+ |\nabla f|^{2})^{\frac{k-2}{2}}\Delta f=0.
\end{align*}
By $\Delta f\le 0$ and \eqref{eq-ratio}, we obtain 
$$\aligned &\int\psi^{k}(1+ |\nabla f|^{2})^{\frac{k-2}{2}}|\Delta f|\\
&\qquad\le 
Cd_x^{-1}\bigg\{\psi^{k-1}(1+ |\nabla f|^{2})^{\frac{k-1}{2}}   
+(k-2)\int\psi^{k-1}(1+ |\nabla f|^{2})^{\frac{k-2}{2}}\bigg\}.\endaligned$$
We include the factor $k-2$ to emphasize that the corresponding term disappears if $k=2$. 
Then, 
\begin{align}\label{eq-iteration2}\begin{split} 
&\int\psi^{k}f(1+ |\nabla f|^{2})^{\frac{k-2}{2}}|\Delta f|\\
&\qquad\le 
CMd_x^{\alpha-1}\bigg\{\int\psi^{k-1}(1+ |\nabla f|^{2})^{\frac{k-1}{2}}   
+(k-2)\int\psi^{k-2}(1+ |\nabla f|^{2})^{\frac{k-2}{2}}\bigg\}.\end{split}\end{align}
Next, take $\varphi= \psi^{k}f(1+ |\nabla f|^{2})^{\frac{k-1}{2}}$ 
in \eqref{eq-interg-by-part}. A similar calculation yields 
\begin{align*}
&\int k\psi^{k-1}f(1+ |\nabla f|^{2})^{\frac{k-2}{2}} \nabla f\nabla\psi  
+n(k-2)\int\psi^{k}(1+ |\nabla f|^{2})^{\frac{k-2}{2}}\\
&\qquad+\int\psi^{k}(1+ |\nabla f|^{2})^{\frac{k-2}{2}}|\nabla f|^2
+(k-1)\int\psi^{k}f(1+ |\nabla f|^{2})^{\frac{k-2}{2}}\Delta f=0.
\end{align*}
Combining with \eqref{eq-iteration2}, we obtain 
\begin{align*} 
&\int\psi^{k}(1+ |\nabla f|^{2})^{\frac{k-2}{2}}|\nabla f|^2\\
&\qquad\le 
CMd_x^{\alpha-1}\bigg\{\int\psi^{k-1}(1+ |\nabla f|^{2})^{\frac{k-1}{2}}   
+(k-2)\int\psi^{k-2}(1+ |\nabla f|^{2})^{\frac{k-2}{2}}\bigg\},\end{align*}
and hence
\begin{align*} 
&\int\psi^{k}(1+ |\nabla f|^{2})^{\frac{k}{2}}\\
&\qquad\le 
CMd_x^{\alpha-1}\bigg\{\int\psi^{k-1}(1+ |\nabla f|^{2})^{\frac{k-1}{2}}   
+(k-2)\int\psi^{k-2}(1+ |\nabla f|^{2})^{\frac{k-2}{2}}\bigg\}.\end{align*}
With the help of  \eqref{eq-iteration1}, a simple iteration yields 
$$\int\psi^{k}(1+ |\nabla f|^{2})^{\frac{k}{2}} \le C M^k d_x^{k\alpha-k + n},$$
and hence
$$\int\psi^{k}|\nabla f|^{k}\le C M^k d_x^{k\alpha -k+ n}.$$
Therefore, for any integer $k\ge 1$, 
$$\int_{B_{d_x}(x)}|\nabla (\psi f)|^k \le C M^k d_x^{k\alpha -k+ n},$$
where $C$ is a positive constant depending only on $n$, $k$ and $\alpha$. 

Next, take any $p\ge 1$. If $p$ is not an integer, by fixing some integer $k>p$, we have
\begin{align*}
\int_{B_{d_x}(x)}|\nabla (\psi f)|^p  & \le 
\left( \int_{B_{d_x}(x)}1\right)^{1-\frac{p}{k}}
\left(\int_{B_{d_x}(x)}(|\nabla (\psi f)|^p)^{\frac{k}{p}}\right)^{\frac{p}{k}}\\
& \leq C M^p d^{n(1-\frac{p}{k})+ [{k\alpha -k+ n} ]\frac{p}{k}  } =C M^p d^{p\alpha  - p +n}.
\end{align*}
Therefore, for any $p\ge 1$, 
\begin{equation}\label{eq-IntegralEstimate}
\|\nabla (\psi f)\|_{L^p(B_{d_x}(x))}\le CMd_x^{\alpha-1+\frac{n}{p}},\end{equation}
where $C$ is a positive constant depending only on $n$, $p$ and $\alpha$. 

If $\alpha\in (0,1)$, we take $p=\frac{n }{1-\alpha}$ so that $\alpha= 1-\frac{n}{p}$. Then, 
$$\|\nabla (\psi f)\|_{L^p(B_{d_x}(x))}\le CM.$$
Hence,  \eqref{f-Lp-int} follows from the  Sobolev embedding.

{\it Step 2.} We now prove \eqref{eq-GradientEstimate}. We first claim 
\begin{align}\label{eq-BoundLaplance}
|\Delta f|\leq \frac{n}{f}(1+|\nabla f|^{2})\quad\text{in }\Omega\cap B_r(x_0).
\end{align}
We fix an $x\in \Omega\cap B_r(x_0)$. 
Since $\Delta f$ and $|\nabla f|$ are invariant under orthogonal transformations, we assume 
$|\nabla f|=f_{1}$ at $x$ by a rotation. Then the equation in  \eqref{eq-main}
reduces to  
\begin{align*}
&\Delta f-\frac{f_{1}^{2}}{1+f_{1}^{2}}f_{11}+\frac{n}{f}=0.
\end{align*} Therefore, at $x$,
\begin{align*}
&\frac{1}{1+f_{1}^{2}}f_{11}+f_{22}+\cdot\cdot\cdot+f_{nn}+\frac{n}{f}=0.
\end{align*}
Since $f_{ii}\leq 0,$ for $i=1,\cdot\cdot\cdot,n,$ we have
$$\frac{1}{1+f_{1}^{2}}\Delta f+\frac{n}{f}\geq0.$$
This implies \eqref{eq-BoundLaplance}
by $\Delta f\leq0$. 

Next, 
\begin{align*}
\Delta(f^{\frac{1}{\alpha}})=\frac{1}{\alpha}f^{\frac{1}{\alpha}-1}\Delta f
+\frac{1}{\alpha}\left(\frac{1}{\alpha}-1\right)f^{\frac{1}{\alpha}-2}|\nabla f|^{2}.
\end{align*} 
By $f\leq Md^{\alpha}$,
we have
\begin{align*}
|\Delta(f^{\frac{1}{\alpha}})|\leq \frac{n}{\alpha^{2}}f^{\frac{1}{\alpha}-2}(1+|\nabla f|^{2})
\quad\text{in }\Omega\cap B_r(x_0).
\end{align*}
Fix an $x\in \Omega\cap B_{r/2}(x_0)$ and a $p>n$. By applying the 
$W^{2,p}$-estimate in $B_{d_x/2}(x)$ and \eqref{eq-IntegralEstimate},  we get 
\begin{align*}
d_x|\nabla(f^{\frac{1}{\alpha}})(x)|&\leq 
C\left\{\frac{n}{\alpha^{2}}d_x^{2-\frac{n}{p}}\|f^{\frac{1}{\alpha}-2}(1+|\nabla f|^{2})\|_{L^{p}(B_{d_x/2}(x))}
+\|f^{\frac{1}{\alpha}}\|_{L^{\infty}(B_{d_x/2}(x))}\right\}\\
&\leq CM^{\frac1\alpha}\left\{d_x^{2-\frac{n}{p}}d_x^{\alpha(\frac{1}{\alpha}-2)}d_x^{2\alpha+\frac np-2}+d_x\right\}
\le CM^{\frac1\alpha}d_x,
\end{align*}
where we used the fact $\alpha\leq 1/2$. Hence, for any $x\in \Omega\cap B_{r/2}(x_0)$, 
$$|\nabla(f^{\frac{1}{\alpha}})(x)|\le CM^{\frac1\alpha}.$$
This is \eqref{eq-GradientEstimate}.
\end{proof}

We now prove a result concerning the local growth. 

\begin{lemma}\label{lemma-growth-rate-1/2}
Let $\Omega\subset \mathbb{R}^{n}$ be a bounded domain
and $f\in C(\bar\Omega)\cap C^\infty (\Omega)$ be a solution of \eqref{eq-main}. 
Suppose $\partial\Omega$ is $C^2$ and $H_{\partial\Omega}>0$ near $x_0\in\partial\Omega$.
Then,
$$f\leq Cd^{\frac{1}{2}}\quad\text{in }\Omega\cap B_r(x_0) ,$$
where $r$ and $C$ are positive constants depending only on $n$, 
the geometry of $\partial \Omega$ near $x_0$  and the diameter of $\Omega$.
\end{lemma}

\begin{proof} Without loss of generality, we assume $x_0$ is the origin and the
$x_n$-direction is the interior normal to $\partial\Omega$. Furthermore, 
we assume 
$\Delta d\le -c_{0}$ in $\Omega\cap B_R$, 
for some positive $c_{0}$ and $R$. 
Set, for some 
$r<R^2/4$,  
$$G_r=\{(x',x_{n})\in \Omega:\,  x'\in B'_{\sqrt{r}},\, 0< d < r\},$$
and, for some $\alpha \in (0,1)$, 
\begin{align*}
w=Ad^{\alpha}+B|x'|^{2},
\end{align*}
where $A$ and $B$ are constants such that 
\begin{equation}\label{eq-choice-A-B}
A\geq\frac{\tau}{r^{\alpha}}|f|_{L^{\infty}}, \quad B=\frac{1}{r}|f|_{L^{\infty}},\end{equation} 
for some large constant $\tau\ge 1$ to be determined. 
Then $f\leq w$ on $\partial G_r.$ 

A straightforward calculation yields 
\begin{align*}
w_{a} & = A\alpha d^{\alpha-1} d_{a}+2Bx_{a},\\
w_{n} & = A\alpha d^{\alpha-1}d_{n},
\end{align*}
and 
\begin{align*}
w_{ab} & = A\alpha(\alpha-1)d^{\alpha-2} d_{a}d_{b}+A\alpha d^{\alpha-1} d_{ab}+2B\delta_{ab},\\
w_{an} & = A\alpha(\alpha-1)d^{\alpha-2} d_{n}  d_{a}+A\alpha d^{\alpha-1}d _{na},\\
w_{nn} & = A\alpha(\alpha-1)d^{\alpha-2}d_{n}^{2}+A\alpha d^{\alpha-1}d_{nn}.
\end{align*}
Then, 
\begin{align*}
|\nabla w|^2&=A^2\alpha^2d^{2\alpha-2}+4AB\alpha d^{\alpha-1}x'\cdot\nabla_{x'}d
+4B^2|x'|^2,\\
\Delta w&=A\alpha(\alpha-1)d^{\alpha-2}+A\alpha d^{\alpha-1}\Delta d+2(n-1)B,\end{align*}
and 
\begin{align*}
w_iw_jw_{ij}&=A^3\alpha^3(\alpha-1)d^{3\alpha-4}
+4A^2B\alpha^2(\alpha-1)d^{2\alpha-3}x'\cdot\nabla_{x'}d\\
&\qquad+2A^2B\alpha^2d^{2\alpha-2}|\nabla_{x'}d|^2
+4AB^2\alpha(\alpha-1)d^{\alpha-2}(x'\cdot\nabla_{x'}d)^2\\
&\qquad+4AB^2\alpha d^{\alpha-1}[2x'\cdot\nabla_{x'}d+x_ax_bd_{ab}]
+8B^3|x'|^2.
\end{align*}
By taking $\alpha\in (0,1)$ and a straightforward calculation, we obtain
\begin{align*}
&A\alpha(\alpha-1)d^{\alpha-2}-\frac{w_iw_j}{1+|\nabla w|^2}w_{ij}\\
&\qquad\le \frac{4AB^2\alpha d^{\alpha-1}[2|x'\cdot\nabla_{x'}d|+|x_ax_bd_{ab}|]}
{1+A^2\alpha^2d^{2\alpha-2}+4AB\alpha d^{\alpha-1}x'\cdot\nabla_{x'}d
+4B^2|x'|^2}.\end{align*}
In fact, the numerator in the left-hand side is given by 
\begin{align*} 
&A\alpha(\alpha-1)d^{\alpha-2}
+4AB^2\alpha(\alpha-1)d^{\alpha-2}[|x'|^2-(x'\cdot\nabla_{x'}d)^2]\\
&\qquad-2A^2B\alpha^2d^{2\alpha-2}|\nabla_{x'}d|^2-8B^3|x'|^2\\
&\qquad-4AB^2\alpha d^{\alpha-1}[2x'\cdot\nabla_{x'}d+x_ax_bd_{ab}].\end{align*}
The first four terms are nonpositive. 
By dropping some positive terms in the denominator and some rearrangements, we get 
\begin{align*}
\Delta w-\frac{w_iw_j}{1+|\nabla w|^2}w_{ij}+\frac{n}{w}
&\le A d^{\alpha-1}\bigg\{\alpha\Delta d+2(n-1)A^{-1}Bd^{1-\alpha}
+\frac{n}{A^2}d^{1-2\alpha}\\
&\qquad+\frac{4(A^{-1}Bd^{1-\alpha})^2[2|x'\cdot\nabla_{x'}d|+|x_ax_bd_{ab}|]}
{\alpha-4A^{-1}B d^{1-\alpha}|x'\cdot\nabla_{x'}d|}\bigg\}.\end{align*}
By the choice of $A$ and $B$ in \eqref{eq-choice-A-B} and $d<r$, we have 
\begin{align*}
\Delta w-\frac{w_iw_j}{1+|\nabla w|^2}w_{ij}+\frac{n}{w}
&\le A d^{\alpha-1}\bigg\{\alpha\Delta d+\frac{2(n-1)}{\tau}
+\frac{n}{\tau^2}r^{2\alpha}|f|_{L^\infty}^{-2}d^{1-2\alpha}\\
&\qquad+\frac{4[2|x'\cdot\nabla_{x'}d|+|x_ax_bd_{ab}|]}
{\tau(\alpha\tau-4|x'\cdot\nabla_{x'}d|)}\bigg\}.\end{align*}
Note $\Delta d\le -c_0$ in $G_r$. We take $\alpha=1/2$. By taking $\tau$ large, we obtain 
$$\Delta w-\frac{w_iw_j}{1+|\nabla w|^2}w_{ij}+\frac{n}{w}\le 0\quad\text{in }G_r.$$
We can apply the maximum principle and obtain 
$$f\leq Ad^{\alpha}+B|x'|^{2}\quad\text{in }G_r.$$ 
By taking $x'=0$, we conclude $f({0},x_{n})\leq Ad^{1/2}$  for any $x_n\in (0,d)$.

In general, we consider 
$$w=Ad^\alpha+B|x'-x_0'|^2,$$ 
and conclude $f(x'_0,x_{n})\leq Ad^{1/2}$  for any $x_n\in (0,d)$.
\end{proof}

We now prove a regularity result in convex domains with singularity. 

\begin{theorem}\label{thrm-local-reg-peiece-mean-conv-dom}
Let $\Omega\subset \mathbb{R}^{n}$ be a bounded convex domain such that, 
in a neighborhood of $x_0\in \partial\Omega$,
$\partial\Omega$  consists of finitely many $C^{2}$-hypersurfaces
$S_i$ intersecting at $x_0$ with $H_{S_i} > 0$, and let $f\in C(\bar\Omega)\cap C^\infty (\Omega)$ 
be the solution of \eqref{eq-main}. Then, $f\in C^{1/2}(\bar\Omega\cap B_r(x_0))$ and 
$$[f]_{C^{1/2}(\bar\Omega\cap B_r(x_0))} \leq C,$$
where $r$ and $C$ are positive constants depending only on 
$n$,  $H_{S_i}$ and the geometry of $\Omega$.
\end{theorem}

\begin{proof} We extend each $S_i$ to form a bounded convex $C^2$-domain $\Omega_i$ 
with $\Omega\subset\Omega_i$ such that 
$(\cap_i\Omega_i)\cap B_R(x_0)=\Omega\cap B_R(x_0)$ and 
$H_{\partial\Omega_i}>0$ on $\partial\Omega_i\cap B_R(x_0)$. 
Let $f_i$ be the solution of \eqref{eq-main} in $\Omega_i$. 
By the maximum principle, we have $f\le f_i$ in $\Omega$. 
By applying Lemma \ref{lemma-growth-rate-1/2} to $f_i$ in $\Omega_i$ near $x_0$
and then restricting to $\Omega$, we have  
$$f \le C d^{{1}/{2}}\quad\text{in }\Omega\cap B_r(x_0),$$ 
where $C$ and $r$ are positive  constants 
depending only on $n$, $H_{\partial\Omega_i\cap B_R}$ and the diameter of $\Omega_i$.
By Theorem \ref{thrm-ConvexDom}, $f$ is concave in $\Omega$. 
Then, we can apply Theorem \ref{thrm-regl-growth-rate} 
and get the desired result.
\end{proof}

Theorem \ref{thrm-reg-peiece-mean-conv-dom} follows easily from 
Theorem \ref{thrm-local-reg-peiece-mean-conv-dom}.

To end this section, we discuss another application of Theorem \ref{thrm-regl-growth-rate}, 
which demonstrates that the H\"older exponent for the regularity can be taken as 
$1/(n+1)$ and $1/i$, for any even integer $i$ between 2 and $n$. 

\begin{theorem}\label{thrm-local-regularity}
Let $\Omega\subset \mathbb{R}^{n}$ be a bounded $C^{n+1,\alpha}$-domain 
with $H_{\partial\Omega}\geq 0$, for some $\alpha\in (0,1)$,
and $f\in C(\bar\Omega)\cap C^\infty (\Omega)$ be the solution of \eqref{eq-main}.
Assume $c_i(x') $ is the first nonzero term in the expansion of $u$ near $0$, 
for some even $i$ between $2$ and $n$, or $i=n+1$. Then, 
$$[f]_{C^{1/i}(\bar\Omega\cap B_R(0))}\le C,$$ 
where $C$ and $R$ are positive constants depending only on $c_i(0)$, 
$n$ and the $C^{n+1, \alpha}$-norm  of $\partial\Omega$ near $0$.
\end{theorem}

Here, for $i=n+1$, $c_{n+1}=c_{n+1, 1}$ if $c_{n+1, 1}(0)$ does not vanish 
and $c_{n+1}=c_{n+1, 0}$ if $c_{n+1, 1}$ vanishes near 0. 

\begin{proof} Let $x_n=\varphi(x')$ be a $C^{n+1, \alpha}$-function 
representing the boundary $\partial\Omega$
near the origin, with $\varphi(0)=0$ and $\nabla\varphi(0)=0$. 
By \eqref{eq-u} and \eqref{eq-f}, we get
\begin{align*}
1 =u_{t}f_{x_{n}},\quad
0 =u_{x'}+u_{t}f_{x'}.
\end{align*}
and
\begin{align*}
0= & u_{tt}f_{x_n}^{2} +u_t f_{x_n x_n},\\
0 =& u_{x_\alpha x_\alpha}+ 2u_{x_\alpha t}f_{x_\alpha}
+ u_{tt}f_{x_\alpha}^{2}+u_{t}f_{x_\alpha x_\alpha}.
\end{align*} 
Hence, $$u_t^2 \Delta f +u_t u_{tt}|\nabla f|^2 
-2\nabla_{x'}u_{t}\nabla_{x'}u+u_{t} \Delta_{x'}u=0.$$
Since $c_i(x') $ is the the first nonzero term in the expansion of $u$, 
by taking $r$ small  depending only on $c_i(0)$, $n$ and 
the $C^{n+1, \alpha}$-norm  of $\varphi$  near $0$, 
we have $\Delta f \le 0$, for  $(x',t) \in B'_r(0) \times(0,r)$.

Next, we verify $f \le C d^{1/i}$ in a neighborhood of $0$. 
By taking $r$ small, we have 
$$|u(x',t)-\varphi(x')|\ge \frac{1}{2}c_i(0) t^{i}\quad\text{for  any }(x',t) \in B'_r(0) \times(0,r).$$ 
Note $x_n=u$ and $t=f$. 
It is easy to verify that  $|x_n-\varphi(x')|\le 2 d(x)$ for $x$ sufficiently small. 
Hence, by taking $r$ sufficiently small and $R = c_i(0)r^{i}/2$, 
we obtain
$$ f\le\left(\frac{4}{c_i(0)} d\right)^{\frac{1}{i}}\quad\text{in }\Omega\cap B_R.$$ 
We then have the desired estimate by Theorem \ref{thrm-regl-growth-rate}.
\end{proof}

\section{An Equivalent Form of the Minimal Surface Equation}
\label{sec-EquivalentEquation}

Let $\Omega$ be a bounded domain and 
$f\in C(\bar\Omega)\cap C^\infty (\Omega)$ be the solution of \eqref{eq-main}. 
Set 
$$w=\frac{1}{4}f^2.$$ Then, $u$ satisfies 
\begin{align}\label{FM1}\begin{split}
\Delta w - \frac{w_i w_j}{w+|D w|^2}u_{ij} +\frac{w}{2 w+ 2|Dw|^2} +\frac{n-1}{2}&=0 \quad\text{in }\Omega,\\
w&=0 \quad \text{on }\partial\Omega,\\
w&>0 \quad\text{in }  \Omega.
\end{split}
\end{align}

The equation in \eqref{FM1} for $n=2$ appears in the study of Chaplygin gas. See the equation (22) 
in \cite{Serre2009}. (The equation in \eqref{eq-main} for $n=2$ is the equation (24) in \cite{Serre2009}.)

Concerning \eqref{FM1}, we have the following global regularity for its solutions. Compare with Theorem 6.1 
in \cite{Serre2009}. 

\begin{theorem}\label{thrm-reg-piece-mean-conv-dom-u}
Let $\Omega\subset \mathbb{R}^{n}$ be a bounded convex domain 
which is the intersection of finitely many bounded convex $C^{2}$-domains $\Omega_i$ 
with $H_{\partial\Omega_i} > 0$, and  let 
$w\in C(\bar\Omega)\cap C^\infty (\Omega)$ be the solution of \eqref{FM1}. 
Then $w\in C^{0,1}(\bar\Omega)$, and
$$|w|_{C^{0,1}(\bar{\Omega})} \leq C,$$
where $C$ is a positive constant depending only on $n$, 
$H_{\partial \Omega_i}$ and the diameter of $\Omega_i$.
\end{theorem}

\begin{theorem}\label{thrm-reg--mean-conv-dom-u}
Let $\Omega\subset \mathbb{R}^{n}$ be a bounded $C^{n+1,\alpha}$-domain 
with $H_{\partial\Omega} > 0$, for some $\alpha\in (0,1)$, and  let 
$w\in C(\bar\Omega)\cap C^\infty (\Omega)$ be the solution of \eqref{FM1}. 
Then, $w\in C^{(n+1)/2}(\bar\Omega)$ if $n$ is even, and
$w\in C^{(n+1-\varepsilon)/2}(\bar\Omega)$ for any $\varepsilon\in (0,1)$ if $n$ is odd. 
In particular, if $n=2$, $w\in C^{1,1/2}(\bar\Omega)$. 
\end{theorem}

In fact, local versions as Theorem \ref{thrm-local-reg-peiece-mean-conv-dom} hold for
both Theorem \ref{thrm-reg-piece-mean-conv-dom-u} and Theorem \ref{thrm-reg--mean-conv-dom-u}. 
Moreover, the regularity in both Theorem \ref{thrm-reg-piece-mean-conv-dom-u} and Theorem \ref{thrm-reg--mean-conv-dom-u}
is optimal in general. Even if the domain $\Omega$ is smooth in Theorem \ref{thrm-reg--mean-conv-dom-u}, 
the regularity of $w$ cannot be improved. 

The proof of Theorem \ref{thrm-reg-piece-mean-conv-dom-u} 
follows from that of Theorem \ref{thrm-local-reg-peiece-mean-conv-dom} by 
employing Theorem \ref{thrm-regl-growth-rate} for $\alpha=1/2$, with 
\eqref{eq-GradientEstimate} replacing \eqref{eq-HolderEstimate}. 

We need to point out that the global Lipschitz property of solutions $u$ of \eqref{FM1} established 
in Theorem \ref{thrm-reg-piece-mean-conv-dom-u} is optimal if domains $\Omega$ admit 
singularity, in which case the solutions $u$ cannot be $C^1$ up to the boundary. 
We demonstrate this by considering $\Omega$ in $\mathbb R^2$. 
Assume that a part of $\partial \Omega$ near $0\in \partial\Omega$ consists of curves $c_{1}$ and $c_{2}$ 
and that the tangent lines $l_1$ and $l_2$ of $c_{1}$ and $c_{2}$ intersect  at $0$ 
with an angle $\alpha \pi$, with $0<\alpha<1.$ 
If $w$ is $C^1$ up to the boundary near 0, then $\nabla w(0)=0$  
by the linear independence of $l_{1}$ and $l_{2}$.  
However by the local expansion of $w$ (see below),  we have 
$$\frac{\partial w}{\partial \nu}=\frac{1}{2H_i}\quad\text{on }c_{i}\setminus \{0\},$$ 
where $H_i$ is the curvature of $c_{i}.$ This leads to a contradiction.

We now discuss briefly the global regularity of solutions of \eqref{FM1} if the boundary $\partial\Omega$ 
is smooth. First, we recall a result established in \cite{HanJiang2014}. 

Set, 
for $n$ even, 
\be\label{b1a}
f_n=a_1\sqrt{d}+a_3(\sqrt{d})^3+\cdots+a_{n-1}(\sqrt{d})^{n-1}
+ a_{n}(\sqrt{d})^n,\ee
and, for $n$ odd, 
\be\label{b1b}f_n=a_1\sqrt{d}+a_3(\sqrt{d})^3+\cdots+a_{n-2}(\sqrt{d})^{n-2} 
+a_{n,1}(\sqrt{d})^n\log \sqrt{d}+ a_{n,0}(\sqrt{d})^n,\ee
where $a_i$ and $a_{i,j}$ are functions on $\partial\Omega$. For example,  
$$
a_1=\sqrt{\frac{2}{H}}.
$$
The following result is a special case in Theorem 7.1 in \cite{HanJiang2014} by taking $k=n$.

\begin{theorem}\label{thrm-Main-F}
Let $\Omega$ be a bounded $C^{n+1,\alpha}$-domain in $\mathbb{R}^n$ with
$H_{\partial \Omega}>0$, for some $\alpha\in (0,1)$, and $(y',d)$ be the 
principal coordinates near $\partial\Omega$. Suppose that 
$f\in C(\bar\Omega)\cap C^\infty(\Omega)$ is a solution of 
\eqref{FM1}. Then, 
there exist functions  $a_i, a_{i,j}\in C^{n-i, \epsilon}(\partial\Omega)$, 
for $i=1, 3, \cdots, n$ and any $\epsilon\in (0,\alpha)$, 
and a positive constant $d_0$ such that, for $f_k$ defined as in \eqref{b1a} or \eqref{b1b}, 
for any $m=0, 1, \cdots, n$, 
and any 
$\epsilon\in (0,\alpha)$, 
$$\partial_{\sqrt{d}}^m(f-f_n)\text{ is $C^{\epsilon}$ in $(y',\sqrt{d})
\in \partial\Omega\times [0,d_0]$},$$ and, 
for any $0<d<d_0$, 
\be\label{eq-MainEstimateHigher}|\partial_{\sqrt{d}}^m(f-f_n)|
\le C(\sqrt{d})^{n-m+\alpha},\ee
where $C$ is a positive 
constant depending only on $n$, $\alpha$ and the $C^{n+1,\alpha}$-norm of 
$\partial\Omega$. 
\end{theorem}

In fact, the remainder $f-f_n$ can be characterized by a multiple integral of multiplicity $n$. 
Then, with such a characterization and the explicit expression of $f_n$ in \eqref{b1a} and \eqref{b1b}, 
Theorem \ref{thrm-reg--mean-conv-dom-u} follows easily. 


\end{document}